\documentclass[12pt]{amsart}
\usepackage{bm}
\usepackage{amsmath, amsthm, amssymb}
\usepackage{graphicx}
\usepackage{mathbbol}
\usepackage{txfonts}
\usepackage[usenames]{color}
\usepackage{pgfplots}
\usepackage{comment}
\usepackage{enumerate}
\allowdisplaybreaks[4]
\usepackage{ circuitikz }

\newtheorem{thm}{Theorem}[section]
\newcommand{\bt}{\begin{thm}}
\newcommand{\et}{\end{thm}}

\newtheorem{cor}[thm]{Corollary}   %remember switch all {coro} to {cor}
\newcommand{\bc}{\begin{cor}}
\newcommand{\ec}{\end{cor}}

\newtheorem{lem}[thm]{Lemma}   %remember to switch all {lem} to {lem}
\newcommand{\bl}{\begin{lem}}
\newcommand{\el}{\end{lem}}

\newtheorem{prop}[thm]{Proposition}
\newcommand{\bp}{\begin{prop}}
\newcommand{\ep}{\end{prop}}

\newtheorem{defn}[thm]{Definition}
\newcommand{\bd}{\begin{defn}}    % This produces an error????
\newcommand{\ed}{\end{defn}}

\newtheorem{rmrk}[thm]{Remark}   %remember to switch all {rmrk} to {rmrk}

\newcommand{\br}{\begin{rmrk}}
\newcommand{\er}{\end{rmrk}}

\newtheorem{example}[thm]{Example}

\newcommand{\GHto}{\stackrel { \textrm{GH}}{\longrightarrow} }

\newcommand{\VFto}{\stackrel {\mathcal{VF}}{\longrightarrow} }

\newcommand{\Scal}{\operatorname{Scalar}}

\newcommand{\be}{\begin{equation}}

 \newcommand{\ee}{\end{equation}}

\newcommand{\diam}{\operatorname{Diam}}

\newcommand{\set}{{\rm{set}}}

\newcommand{\vol}{\operatorname{Vol}}

\newcommand{\Sph}{{\mathbb S}}         %unit sphere%

\def\({\left(}
\def\){\right)}
\def\d{\delta}

\pgfplotsset{compat=1.18} 

\begin{document}

\title[Geometric Convergence to an Extreme Limit Space]{Geometric Convergence to an Extreme Limit Space with nonnegative scalar curvature}

\author{Christina Sormani}
\thanks{Prof. Sormani is partially supported by NSF DMS \#1006059.}
\address{Christina Sormani, CUNY Graduate Center, NY, NY 10016
and Lehman College, Bronx NY 10468}
\email{sormanic@gmail.com}

\author{Wenchuan Tian}
\address{Wenchuan Tian, Department of Mathematics, University of California, Santa Barbara, 
Santa Barbara, CA 93106-3080}
\email{tian.wenchuan@gmail.com}

\author{Wai-Ho Yeung}
\address{Wai-Ho Yeung, Department of Mathematics, 
University of Vienna, Austria 
}
\email{a12519833@unet.univie.ac.at}
\keywords{Scalar Curvature, Intrinsic Flat Convergence}

\begin{abstract}
In 2014, Gromov conjectured that sequences of manifolds with nonnegative scalar curvature should have subsequences which converge in some geometric sense to limit spaces with some notion of generalized nonnegative scalar curvature.  
In recent joint work with Changliang Wang, the authors found a sequence of warped product Riemannian metrics on $\Sph^2\times \Sph^1$ 
with nonnegative scalar curvature
 whose metric tensors converge in the $W^{1,p}$  sense for $p<2$ to an extreme warped product limit space where the warping function hits infinity at two points.  Here we study this extreme limit space as a metric space and as an integral current space and prove the sequence converges in the volume preserving intrinsic flat and measured Gromov-Hausdorff sense to this space.   This limit space may now be used to test any proposed definitions for generalized nonnegative scalar curvature.   One does not need expertise in Geometric Measure Theory or in Intrinsic Flat Convergence to read this paper.   
\end{abstract}

\maketitle

%\tableofcontents

\section{Introduction}

Gromov introduced the idea of geometric convergence for a sequence of compact Riemannian manifolds, $(M_j,g_j)$, by viewing them as metric spaces, $(M_j, d_j)$, where $d_j$ is the Riemannian distance, and finding limits that are metric spaces, $(M_\infty, d_\infty)$ \cite{Gromov-metric}.   He defined the Gromov-Hausdorff (GH) distance between metric spaces, by taking an infimum over all possible distance preserving maps, $\varphi_j:M_j \to Z$, over all possible compact metric spaces, $Z$, of the Hausdorff distance between the images in $Z$:
\be
d_{GH}((X_1,d_1), (X_2,d_2))=\inf \, d_H^Z(\varphi_1(X_1), \varphi_2(X_2)).
\ee
Note that if one has a sequence of distances, $d_j$, all defined on the same compact
metric space, $X$, and $d_j$ converge uniformly to $d_\infty$ as functions, then
$(X, d_j) \GHto (X,d_\infty)$.   Gromov-Hausdorff convergence has the advantage of
not requiring all the spaces to be the same.
Gromov proved a beautiful compactness theorem for GH convergence as well \cite{Gromov-metric}.

Gromov-Hausdorff convergence has been applied with great success to study sequences of Riemannian manifolds with nonnegative sectional curvature. The GH-limits of such sequences are Alexandrov spaces which are metric spaces satisfying a generalized notion of nonnegative sectional curvature.   See the work of Burago-Gromov-Perelman \cite{BGP} and others.
Gromov-Hausdorff convergence has also been applied combined with measure convergence
to study sequences of Riemannian manifolds with nonnegative Ricci curvature.   The measured GH-limits of such sequences are $CD$ and $RCD$ spaces which are metric spaces with Borel measures satisfying generalized notion of nonnegative Ricci curvature.  See the work of Cheeger-Colding, Lott-Villani, Sturm, and
Ambrosio-Gigli-Savare in
\cite{AGS} \cite{ChCo-PartI}  \cite{Lott-Villani}
\cite{Sturm}.

In \cite{Gromov-Plateau} and \cite{Gromov-Dirac}, Gromov vaguely conjectured that a sequence of Riemannian manifolds with nonnegative scalar curvature, $\Scal \ge 0$, should have a subsequence which converges in some geometric sense to a limit space with some generalized notion of ``nonnegative scalar curvature".  He suggested that the intrinsic flat convergence defined by Sormani and Wenger in \cite{SW-JDG} building upon work of Ambrosio-Kirchheim in \cite{AK} might be the appropriate notion.   Various conjectures have been formulated in this direction and related examples have been constructed.
See the survey by Sormani \cite{Sormani-conjectures} and papers that cite this survey.   

It is not yet known what notion of generalized nonnegative scalar curvature works best for the geometric limits of Riemannian manifolds with nonnegative scalar curvature.   In \cite{STW-Extreme}, Sormani-Tian-Wang constructed a sequence of warped product manifolds with nonnegative scalar curvature with fibres that stretch to infinite length (see Example~\ref{ex:sequence} below).
They proved this sequence converges smoothly away from the
singular fibres to an extreme warped product space
(see Example~\ref{ex:limit} below).  In this paper we
prove that this sequence converges in a geometric way and
describe the limit metric space (see Theorem~\ref{Thm:Main} and Theorem~\ref{thm:VF} within).
This specific limit space can then be used to test various notions of nonnegative scalar curvature.

The sequence of Riemannian warped product manifolds that we explore in this paper was first introduced Sormani-Tian-Wang in \cite{STW-Extreme}:

\begin{example}\label{ex:sequence}
Consider the sequence of warped product Riemannian manifolds, $\Sph^2\times_{g_j} \Sph^1$, which are diffeomorphic to $\Sph^2\times \Sph^1$, with Riemannian metrics 
\be
g_j=g_{\Sph^2}+f_j^2(r) g_{\Sph^1}=dr^2+ \sin^2(r) d\theta^2+f_j^2(r) d\varphi^2.
\ee
written using fixed $(r,\theta)$ coordinates on $\Sph^2$ and fixed $\varphi$ on the $\Sph^1$-fibres
where 
\be
f_j(r)=\ln\left(\frac{1+a_j}{\sin^2 r+a_j}\right)+\beta,
\ee
taking fixed $\beta\geq 2$, positive $a_j$ decreasing to $0$.   
 \end{example}

\begin{figure}[h] %  figure placement: here, top, bottom, or page
   \centering
   \includegraphics[width=3in]{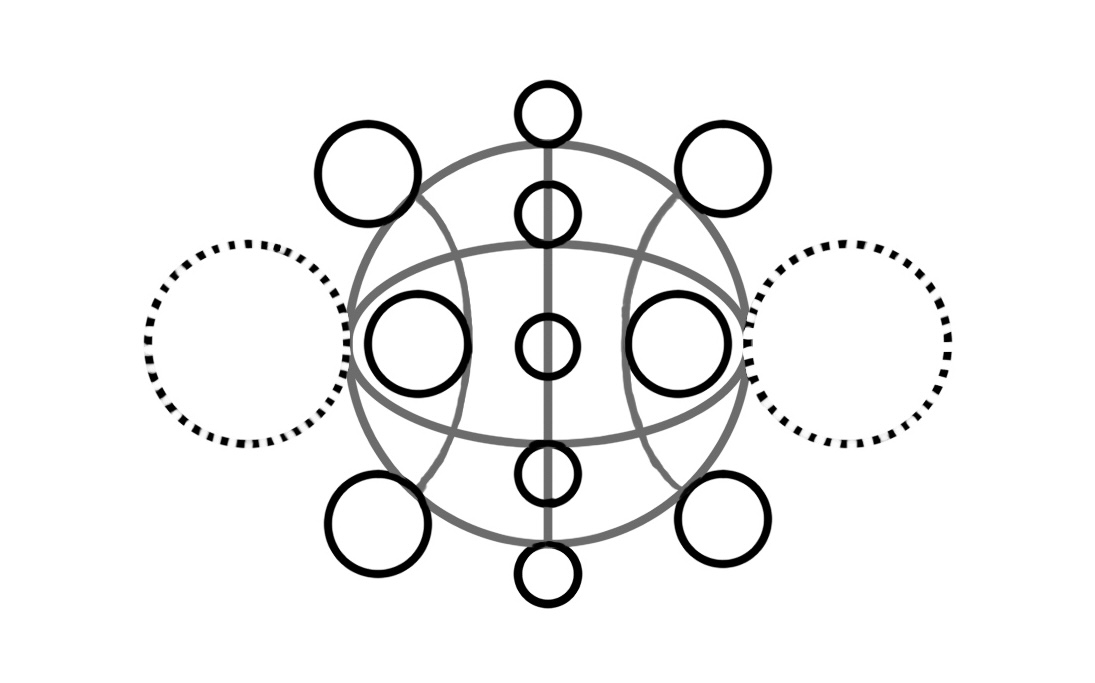} 
   \caption{ Example~\ref{ex:limit} is an extreme limit warped product space, $(\Sph^2\times \Sph^1,g_\infty)$, 
   in which the circular fibers, $r^{-1}(0)$ and $r^{-1}(\pi)$, have been stretched infinitely.  Since the lengths of any curves in these fibers are infinite, we depict these fibers as circles of discrete points.   Away from these
   two singular fibers, $(\Sph^2\times \Sph^1\setminus S,g_\infty)$ is a smooth Riemannian manifold.
   }
   \label{fig:STW-intro}
\end{figure}
 
 In \cite{STW-Extreme}, it was proven that this sequence has $\Scal_j \ge 0$ yet as $j\to \infty$ with diameter and volume bounded above. In addition, closed minimal surfaces in these manifolds cannot have area decreasing to $0$ as $j\to \infty$.  
   
In addition Sormani-Tian-Wang proved smooth convergence of the $f_j \to f_\infty$ on compact sets away from the
set
\be
S=r^{-1}(0) \cup r^{-1}(\pi)
\ee
and defined the following extreme limit space
using this limit function.

\begin{example}\label{ex:limit}
Consider the extreme warped product space, $\Sph^2\times_{f_\infty}\Sph^1$, 
with metric tensor
\be
g_\infty=g_{\Sph^2}+f_\infty^2(r) g_{\Sph^1}=dr^2+ \sin^2(r) d\theta^2+f_\infty^2(r) d\varphi^2.
\ee
written using $(r,\theta)$ coordinates on $\Sph^2$ and $\varphi$ on the $\Sph^1$-fibres with the extreme warping function defined by
\be \label{f-infty}
f_\infty(r, \theta)=\ln \left(\frac{1}{\sin^2 r}\right)+\beta=-2\ln \sin r+\beta,
\ee
where $\beta\geq 2$ as in Example \ref{ex:sequence}.   Note that $g_\infty$ is smooth away from the singular set
\be
\label{DefinitionSingularSet}
S=\{(r,\theta,\varphi)\in\Sph^2\times \Sph^1,\text{ such that } r=0\text{ or }r=\pi\}
\ee
so $(\Sph^2\times \Sph^1\backslash S, g_\infty)$ is a smooth open Riemannian manifold.
See Figure~\ref{fig:STW-intro}.
 \end{example}

In addition to proving $g_j$ converge to $g_\infty$
smoothly away from the singular set, $S$,
Sormani-Tian-Wang proved 
that the metric tensors, $g_j$, converge to $g_\infty$  in $W^{1,p}(\Sph^2\times \Sph^1)$ as $j\to \infty$ for  $p\in [1,2)$. They also proved 
the limit space has nonnegative distributional scalar curvature in the sense defined by D.Lee-LeFloch \cite{Lee-LeFloch} building upon the work of LeFloch-Mardare \cite{LM07}.  

In this paper, we study the geometric notions of convergence of this sequence and the geometric properties of the limit space.  

\begin{thm}\label{Thm:Main}
The sequence of warped product Riemannian manifolds in Example~\ref{ex:sequence} 
viewed as metric spaces,
$(\Sph^2\times \Sph^1,d_j)$, converges in the uniform and Gromov-Hausdorff sense to the metric space
$(\Sph^2\times \Sph^1,d_\infty)$ 
which 
is the metric completion of
the extreme limit space,
$(\Sph^2\times \Sph^1\setminus S,d_{g_\infty})$, of Example~\ref{ex:limit}.  This limit space is
a compact
metric space homeomorphic to the
isometric product space,
$(\Sph^2\times \Sph^1,d_0)$.
In addition the singular set $S$
inside this limit space has Hausdorff dimension $1$ but the
singular fibres are unrectifiable,
$\mathcal{H}^1_{d_\infty}(S)=\infty$.
Finally $
{\mathcal{H}}^3_{d_\infty}(\Sph^2\times \Sph^1, d_\infty)=(2\pi)^2(2\beta+4-2\ln 4).
$
\end{thm}

We also prove volume preserving intrinsic flat convergence of our sequence to a subset of the Gromov-Hausdorff limit whose metric completion is the Gromov-Hausdorff limit.  See Theorem~\ref{thm:VF} stated within.
These theorems allow others to then test various geometric definitions of generalized nonnegative scalar curvature on this limit space.   
     
Note that this extreme limit space is not an Alexandrov Space with curvature bounded from below, nor is it $CD(3,K)$-space, nor an RCD space since the
Ricci curvature of $g_\infty$ is not bounded from below on $\Sph^2\times \Sph^1\backslash S$ as shown in \cite{STW-Extreme}.  It is outside the class of $C^0$
Riemannian manifolds studied by Bamler and Burkhardt-Guim \cite{Bamler-Gromov} \cite{Burkhardt_Guim_2019} and by Gromov and Chao Li in \cite{Gromov-Dirac}\cite{li2019positive} but
still might be studied using their notions of generalized scalar curvature.  It has lower regularity than the spaces studied by D.Lee-LeFloch \cite{Lee-LeFloch} and by M.C.Lee-Topping
\cite{lee2022metric}.   
It is also outside of the class
of Riemannian manifolds with isolated singularities studied by Li-Mantoulidis \cite{li2019positive}, because the singular set $S$ does not consist of isolated points.  We hope it will
provide new insight into the full class of
possible limit spaces.

This paper is written to be easily read by those who have not studied geometric notions of convergence before. We begin with a background section reviewing the properties of Example~\ref{ex:sequence}.  In Section 3, we
prove 
pointwise convergence of the distance functions, $d_j$ to $d_\infty$  and prove the claimed homeomorphism.   In Section 4 we review GH convergence and prove uniform and GH convergence of
$(\Sph^2\times\Sph^1,d_j)$
to $(\Sph^2\times\Sph^1,d_\infty)$ applying a
theorem of Perales-Sormani from \cite{PS-Monotone}.
In Section 5 we prove the
metric completion of the smooth limit away from the singular set is $(\Sph^2\times\Sph^1,d_\infty)$ and control the Hausdorff measures.   At the end of this section we complete the proof of Theorem~\ref{Thm:Main}.  In Section 6, we state and prove intrinsic flat convergence relying on readers to consult the work of Perales-Sormani for background \cite{PS-Monotone}.   In Section 7 we discuss open problems.

\tableofcontents

\section{Background}
\label{Sect-Seq}

In this section we review the properties of
the sequence of Riemannian warped product manifolds of Example~\ref{ex:sequence}.

\subsection{Warping functions
of Example~\ref{ex:sequence}}

The following was proven by
Sormani-Tian-Wang in \cite{STW-Extreme}:

\begin{lem} \label{lem:warp-inc}
For $f_j(r)$ and $g_j$ as defined in Example \ref{ex:sequence} we have
\be
f_j(r) \le f_{j+1}(r)
\ee
and, for all tangent vectors, $V$,
we have
\be 
g_j(V,V) \le  g_{j+1}(V,V).
\ee
\end{lem}

\begin{lem} \label{lem:warp-unbounded}
The sequence $f_j$ as defined in Example \ref{ex:sequence} is unbounded in the sense that
\be
\lim_{j\to \infty} f_j(0) \to \infty \textrm{ and } \lim_{j\to \infty} f_j(\pi) \to \infty.
\ee
\end{lem}

\begin{prop} \label{prop:warp-conv}
The sequence $g_j$ as defined in Example \ref{ex:sequence} 
converges smoothly away from $S$ to $g_\infty$ of Example~\ref{ex:limit}.   That is, on compact
subsets $K \subset (\Sph^2\times\Sph^1)\setminus S$ we have
\be
|f_j-f_\infty|_{C^\infty(K)}\to 0 \textrm{ and } |g_j-g_\infty|_{g_\infty,C^\infty(K)}\to 0
\ee
\end{prop}

\subsection{Volumes of Example~\ref{ex:sequence}}

The following proposition was proven by Sormani-Tian-Wang in \cite{STW-Extreme}:

\begin{prop}[Volumes]\label{prop-vol-sequence}
For $\Sph^2\times_{f_j}\Sph^1$ as defined in Example \ref{ex:sequence}, 
the volumes are uniformly bounded from above:
\be
\vol_j(\Sph^2\times_{g_j} \Sph^1) \le 4\pi^3\beta.
\ee
\end{prop}

\subsection{Lengths on Example~\ref{ex:sequence}}

As for any Riemannian manifold, the length of a curve is defined by:

\begin{defn}[Length of a Curve]\label{DefinitionLengthCurve}
Given a continuously differentiable curve $c: [a,b]\to \Sph^2\times \Sph^1$, we define the length of $c$ as measured by the metric $g$ as
\be
L_g(c)=\int_{a}^b\sqrt{g(\dot{c},\dot{c})}\,dt,
\ee
where $\dot{c}$ denotes the derivative with respect to $t$.  
For the sequence, $g_j$, of Example~\ref{ex:sequence}, we will write, $L_j=L_{g_j}$. 
\end{defn}

Sormani-Tian-Wang proved the following lemma in \cite{STW-Extreme}:

\begin{lem}\label{lem-fibres}
The $\Sph^1$ fibres above the poles in Example~\ref{ex:sequence} parametrized by
the curves,
\be
c_0(t)=(0, 0, t) \textrm{ and }
c_\pi(t)=(\pi, 0, t),
\ee
where $t\in [0, 2\pi]$ 
have
lengths, $L_j(c_0)=2\pi f_j(0)$ and $L_j(c_\pi)=2\pi f_j(\pi)$,
which diverge to infinity as $j \to \infty$.
\end{lem}

\subsection{Distances on Example~\ref{ex:sequence}}

On any Riemannian manifold the distance function is
defined as follows:

\begin{defn}[Distance Function]\label{defn:dj}
Given any Riemannian manifold, $(M,g)$,
	For any $p,\ q\in M$, define
	\be\label{eq:d_g}
	d_{g}(p,q)=\inf\{L_{g}(c)\},
	\ee
	where the infimum is taken over continuously differentiable curves, $c:[0,1]\to M$
    with $c(0)=p$ and $c(1)=q$ and 
    and where $L_g$ is defined as in Definition~\ref{DefinitionLengthCurve}.

    For Example~\ref{ex:sequence},
    $(\Sph^2\times\Sph^1,g_j)$
    is a smooth Riemannian manifold and
    we define $d_j=d_{g_j}$ as in (\ref{eq:d_g})
    using curves in $\Sph^2\times\Sph^1$.
    For Example~\ref{ex:limit},
    $(\Sph^2\times\Sph^1\setminus S,g_\infty)$
    is a smooth Riemannian manifold
    and we define $d_{g_\infty}$ as in (\ref{eq:d_g})
    using curves that stay
    within $\Sph^2\times\Sph^1\setminus S$
    avoiding the singular set, $S$.
    \footnote{We will define $d_\infty$ later 
    in Theorem~\ref{thm:ptwise-homeo} as the
    pointwise limit of the $d_j$.}
\end{defn}

Sormani-Tian-Wang proved the following three results in \cite{STW-Extreme}:

\begin{lem}[Distance Functions are Increasing]\label{LemDistanceFunctionsInc}
 The sequence of distance functions, $d_j$, as defined in Definition \ref{defn:dj} 
 satisfies the inequality:
 	\be
 	d_j(p,q)\leq d_{j+1}(p,q) \qquad \forall j \in {\mathbb N}.
 	\ee
\end{lem}

\begin{lem}[Uniform Bound for Distance Functions]\label{lem:bound-dj-LE}
Given any $p_1=(r_1,\theta_1,\varphi_1)$ and $p_2=(r_2,\theta_2,\varphi_2)$ in $\Sph^2\times_{f_j}\Sph^1$
as defined in Example \ref{ex:sequence}, 
	\be
d_j(p_1,p_2) \leq 
		|r_1-r_2|+\sin (r_2)\,d_{{\mathbb S}^1}(\theta_1,\theta_2)+f_j(r_2) \,d_{{\mathbb S}^1}(\varphi_1,\varphi_2).
\ee
\end{lem}

\begin{prop}[Uniform Bound for Diameter]\label{prop:diamj}
	For $\Sph^2\times_{f_j}\Sph^1$ as defined in Example \ref{ex:sequence}, 
	 we have 
	 \be
\diam(\Sph^2\times_{f_j}\Sph^1)
\le D_0=(3+2\beta)\pi.
	 \ee
\end{prop}

\subsection{The Isometric Product Space}

The isometric product of a sphere and a circle is defined as follows:

\begin{example}\label{ex:isom-prod}
The isometric product tensor
on $\Sph^2\times \Sph^1$ is
\be
	g_0=g_{\Sph^2}+ g_{\Sph^1}=dr^2+ \sin^2(r) d\theta^2+d\varphi^2,
	\ee
	written using $(r,\theta)$ coordinates on $\Sph^2$ and $\varphi$ on the $\Sph^1$-fibres.
The isometric product distance,
\be
	d_0(p,q)=d_{g_0}(p,q)=\inf\{L_{g_0}(c)\},
\ee
defined as in
Definition~\ref{defn:dj}.
\end{example}

The following lemma follows easily from
Definition~\ref{DefinitionLengthCurve},
Definition~\ref{defn:dj},
the definition of $g_0$ in
Example~\ref{ex:isom-prod},
and the definition of $g_j$ 
with $f_j\ge \beta\ge 1$ in
Example~\ref{ex:sequence}:

\begin{lem}\label{lem:d_j-ge-d_0}
The distances on Example~\ref{ex:sequence}
and the distances on the isometric product
in Example~\ref{ex:isom-prod}
satisfy
\be
d_j(p,q)\ge d_0(p,q) \quad \forall 
p,q\in \Sph^2\times \Sph^1.
\ee
\end{lem}

The following two lemmas are well known:

\begin{lem}\label{lem:iso-Pyth}
The isometric product space
on $(\Sph^2\times \Sph^1, d_0)$
satisfies the Pythagorean Theorem:
 \be\label{eq-isom-sqrt}
 d_0((r_p,\theta_p,\varphi_p),(r_q,\theta_q,\varphi_q))=\sqrt{d_{\Sph^2}((r_p,\theta_p),(r_q,\theta_q))^2 + d_{\Sph^1}(\varphi_p,\varphi_q)^2}.
 \ee
 so
 \be \label{eq-isom-r}
 d_0((r_p,\theta_p,\varphi_p),(r_q,\theta_q,\varphi_q)) \ge d_{\Sph^2}((r_p,\theta_p),(r_q,\theta_q))\ge |r_p-r_q|
 \ee
 and
 \be \label{eq-isom-phi}
 d_0((r_p,\theta_p,\varphi_p),(r_q,\theta_q,\varphi_q)) \ge d_{\Sph^{1}}(\varphi_p,\varphi_q).
 \ee
\end{lem}

\begin{lem}\label{lem:iso-cmpct}
The isometric product space
on $(\Sph^2\times \Sph^1, d_0)$
is compact.
\end{lem}

The following new lemma will be useful to us later:

 \begin{lem}
\label{LemmaThetaInequality}
Fix $p=(r_p,\theta_p,\varphi_p)\in (\Sph^2 \times \Sph^1)\backslash S$ and choose any
\be
{\delta\in \left(0,\min\left\{\tfrac{r_p}{2}, \tfrac{\pi-r_p}{2}\right\}\right)}
\ee
For any $q\in \Sph^2 \times \Sph^1$ such that $d_0(p,q)<\delta$, we have
\be
d_{\Sph^1}(\theta_p,\theta_q)<\frac{\delta}{\sin (r_p/2)}.
\ee
\end{lem}

\begin{proof}
By symmetry, we may assume  $r_p\in (0,\pi/2]$.
Let $c:[0,1]\to (\Sph^2\times \Sph^1, g_{\Sph^2}+g_{\Sph^1})$ be the geodesic connecting $p$ and $q$. In coordinate we have 
\be
c(t)=(r(t),\theta(t),\varphi(t))\text{ for }t\in [0,1].
\ee

Next, we prove that 
\be \label{eq:next}
r(t) \in (r_p/2,3r_p/2)
\subset (r_p/2,\pi-r_p/2)\subset (0,\pi)
\text{ for }t\in [0,1].
\ee
By (\ref{eq-isom-r}), we have for $t\in[0,1]$
\be
|r_p-r(t) |\le d_{\Sph^2}((r_p,\theta_p),(r(t),\theta(t)))\le d_0(p,c(t))<\delta,
\ee
which by $\delta<r_p/2$ and the triangle inequality implies
\be
r_p/2 <r_p-\delta< r(t)< r_p+\delta<3r_p/2.
\ee
Since $r_p\in (0, \pi/2]$ we have
$$
3r_p/2\le 3\pi/4\le \pi-r_p,
$$
so we have
our claim in (\ref{eq:next}).

By (\ref{eq:next}) and the concavity and symmetry of the sine function
on $(0,\pi)$, we have
\be\label{InequalityRq}
\sin r(t) \geq \left(\sin \left(\tfrac{r_p}{2}\right)
+\sin \left(\pi -\tfrac{r_p}{2}\right) \right)/2
=\sin \left(\tfrac{r_p}{2}\right)
\textrm{ for } t\in[0,1].
\ee

Through direct calculation we have
\begin{equation}
\begin{split}
d_{\Sph^1}(\theta_p,\theta_q) &\leq \int_0^1 |\theta'(t)| \,dt
\leq \frac{1}{\sin (r_p/2)}\int_0^1 \sin r(t) |\theta'(t)| \,dt
\\
&\leq \frac{1}{\sin (r_p/2)}\int_0^1 \sqrt{(r'(t))^2+(\sin r(t) \theta'(t))^2 +(\varphi'(t))^2}dt\\
&= \frac{d_0 (p,q)}{ \sin (r_p/2)}\leq \frac{\delta}{\sin (r_p/2)}
\end{split}
\end{equation}
which completes the claim.
\end{proof}

%%%%%%%%%%%

\section{Pointwise Convergence and Homeomorphism}

In this section, we introduce the metric
space, $(\Sph^2\times \Sph^1,d_\infty)$,
and prove it is homeomorphic to the isometric
product space in Theorem~\ref{thm:ptwise-homeo}.
We will later prove in Theorem~\ref{thm:GH}
that this metric
space is the $GH$ limit
of $(\Sph^2\times \Sph^1,d_j)$.
Recall that $d_j=d_{g_j}$ are the Riemannian
distances as in Definition~\ref{defn:dj}
of the sequence of Riemannian manifolds 
in Example~\ref{ex:sequence}.

\begin{thm}\label{thm:ptwise-homeo}
There is a metric space $(\Sph^2\times \Sph^1,d_\infty)$
defined by taking the pointwise
limit of the monotone increasing sequence
\be \label{eq:mono-lim}
d_\infty(x,y)=\lim_{j\to \infty}d_j(x,y) =
\sup_{j\in {\mathbb N}} d_j(x,y)
\ee 
where $d_j$ are the distance functions,
of the warped product Riemannian manifolds in Example~\ref{ex:sequence} as in Definition~\ref{defn:dj}. 
Furthermore this 
pointwise limit space,
$(\Sph^2\times \Sph^1,d_\infty)$,
is homeomorphic to the
isometric product space,
$(\Sph^2\times \Sph^1,d_0)$,
of Example~\ref{ex:isom-prod}.
\end{thm}

\begin{rmrk}\label{rmrk:d-vs-g}
We
do not yet know how $d_\infty$ 
of Theorem~\ref{thm:ptwise-homeo}
is
related to $g_\infty$ of Example~\ref{ex:limit}.
We will see this relationship
later when we prove
Theorem~\ref{thm:completion}.
\end{rmrk}

We break the proof of Theorem~\ref{thm:ptwise-homeo}
into two propositions that are proven within this section.
First we prove pointwise convergence in Proposition~\ref{prop:ptlim-exists}.  Then we prove various useful lemmas about $d_\infty$
and a couple of
general technical lemmas.
Finally, we prove the homeomorphism in 
Proposition~\ref{prop:ptwise-homeo} using these lemmas.

\subsection{Pointwise Convergence}

First we show that the pointwise limit
in Theorem~\ref{thm:ptwise-homeo}
exists and is a definite metric space:

\begin{prop}\label{prop:ptlim-exists}
The sequence of
Riemannian
distance functions, $d_j=d_{g_j}$,
of Example~\ref{ex:sequence} defined as in Definition~\ref{defn:dj}
are monotone increasing and
converge pointwise as in (\ref{eq:mono-lim})
to a distance
function,
$d_\infty$, on
$\Sph^2\times \Sph^1$
.
\end{prop}

\begin{proof}
Fix $p,q\in \Sph^2\times \Sph^1$.
We have already shown the
$d_j(p,q)$ is a monotone increasing sequence of real numbers in Lemma~\ref{LemDistanceFunctionsInc} that is uniformly bounded above by the uniform upper bound, $D_0$, on the diameter of these spaces as seen in Proposition~\ref{prop:diamj}.
By the Monotone Convergence Theorem of sequences of real numbers we have (\ref{eq:mono-lim}).   

To see that $(\Sph^2\times \Sph^1,d_\infty)$ defines a definite metric space of diameter $\le D_0$, see, for example, Lemma 2.5 of \cite{PS-Monotone} by Perales-Sormani.
\end{proof}

 \subsection{Estimates on Distances}
  
In this section we prove some useful bounds on the
pointwise limit, $d_\infty$, of Theorem~\ref{thm:ptwise-homeo} that we will apply in a few places later.   See Figure~\ref{fig:STW-le}.

\begin{figure}[h] %  figure placement: here, top, bottom, or page
   \centering
\includegraphics[width=3in]{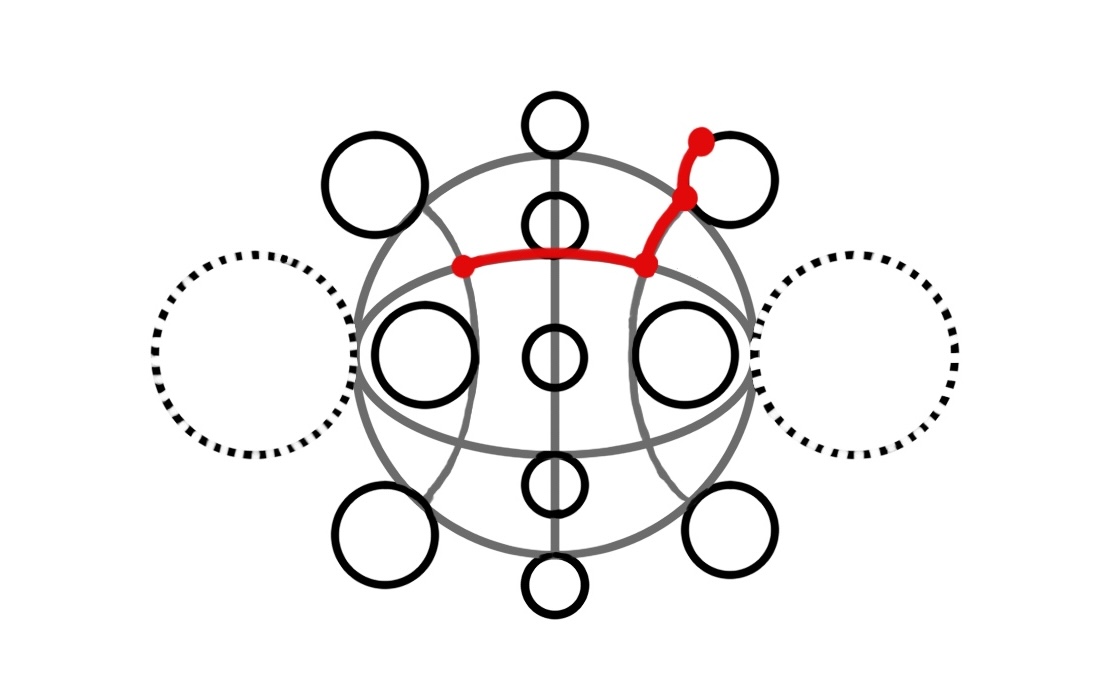} 
   \caption{ The pointwise limit space, $(\Sph^2\times \Sph^1, d_\infty)$, 
   is a metric space in which the distances between
   points, even those in
   the circular fibers, $r^{-1}(0)$ and $r^{-1}(\pi)$, can be estimated
   using Lemma~\ref{LemmaDinfty-cLE}.
   }
   \label{fig:STW-le}
\end{figure}

\begin{lem} \label{LemmaDinfty-cLE}
Given any point $p_1=(r_1,\theta_1,\varphi_1)\in \Sph^2\times \Sph^1$ and any point $p_2=(r_2,\theta_2,\varphi_2)\in \Sph^2\times \Sph^1\setminus S$ so that $r_2\in (0,\pi)$,
we have 
	\be
d_\infty(p_1,p_2) \leq 
		|r_1-r_2|+\sin (r_2)\,d_{{\mathbb S}^1}(\theta_1,\theta_2)+f_\infty(r_2) \,d_{{\mathbb S}^1}(\varphi_1,\varphi_2).
\ee
\end{lem}

\begin{proof}
Recall that in \cite{STW-Extreme}, Sormani-Tian-Wang estimated $d_j(p_1,p_2)$ using Definition~\ref{defn:dj} and a specific curve running between the points
as depicted in Figure~\ref{fig:STW-le}.
The curve is the concatenation of three curves:
\be
d_j(p_1,p_2)\le L_{g_j}(C_1)+L_{g_j}(C_2)
+L_{g_j}(C_3)
\ee
where the curve, $C_1$,
runs radially from
$p_1=(r_1,\theta_1,\varphi_1)$ to a point
$(r_2,\theta_1,\varphi_1)$
, and then $C_2$ runs 
in the $\theta$ direction
to $(r_2,\theta_2,\varphi_1)$ and then $C_3$
runs around the fiber to $p_2=(r_2,\theta_2,\varphi_2)$.
Estimating these lengths using
\be
g_j=dr^2+\sin^2(r)d\theta^2+f_j^2(r)d\varphi^2,
\ee
they got
\be \label{eq:here-bound-dj-LE}
d_j(p_1,p_2) \leq 
		|r_1-r_2|+\sin (r_2)\,d_{{\mathbb S}^1}(\theta_1,\theta_2)+f_j(r_2) \,d_{{\mathbb S}^1}(\varphi_1,\varphi_2).
\ee
By Lemma~\ref{lem:warp-inc}, we know
$f_j(r_2) \le f_\infty(r_2)$.  Substituting this into
(\ref{eq:here-bound-dj-LE}),
we have
\be
d_j(p_1,p_2) \leq 
		|r_1-r_2|+\sin (r_2)\,d_{{\mathbb S}^1}(\theta_1,\theta_2)+f_\infty(r_2) \,d_{{\mathbb S}^1}(\varphi_1,\varphi_2).
\ee
Our lemma follows by taking the pointwise limit as $j\to\infty$ of this last inequality.
\end{proof}

If we have a pair of points which both lie in the singular set, or are close to the singular set, we use the triangle inequality to
estimate the distance between them:

\begin{lem} \label{LemmaDinfty-cLE2}
Given any pair of points, $p=(r_p,\theta_p,\varphi_p)$ and $q=(r_q,\theta_q,\varphi_q)$ in $\Sph^2\times \Sph^1$ and any $r_0\in (0,\pi)$,
we have
\be
d_\infty(p,q)\le
|r_p-r_0|+|r_q-r_0|+
\sin (r_0)\,
d_{{\mathbb S}^1}(\theta_p,\theta_q)+
f_\infty(r_0)\,
d_{{\mathbb S}^1}(\varphi_1,\varphi_2)
\ee
\end{lem}

\begin{proof}
Given any pair of points, $p=(r_p,\theta_p,\varphi_p)$ and $q=(r_q,\theta_q,\varphi_q)$ in $\Sph^2\times \Sph^1$.
Let $p'=(r_0,\theta_p,\varphi_p)$ and $q'=(r_0,\theta_q,\varphi_q)$. 
By Lemma~\ref{LemmaDinfty-cLE} applied three times
we have,
\begin{eqnarray}
d_\infty(p,p') &\leq &
		|r_p-r_0|+0+0,
        	\\
d_\infty(p',q') &\leq & 
		0+\sin (r_0)\,d_{{\mathbb S}^1}(\theta_p,\theta_q)+f_\infty(r_0)\,
        d_{{\mathbb S}^1}(\varphi_1,\varphi_2),\\
d_\infty(q,q')&\leq &
		|r_q-r_0|+0+0 \,\,\,\le\,\,\, \pi/2.
\end{eqnarray}
By the triangle inequality,
$d_\infty(p,q)$ is bounded above
by the sum of these distances.
\end{proof}

\begin{prop}\label{prop:diam-infty}
The diameter of the pointwise limit metric space
satisfies:
	 \be
\diam_\infty(\Sph^2\times\Sph^1)\leq (3+2\beta)\pi.
	 \ee
\end{prop}

\begin{proof}
This follows from Lemma~\ref{LemmaDinfty-cLE2}
taking $r_0=\pi/2$ and using 
\be
|r-\pi/2|\le \pi/2 \qquad \forall r\in [0,\pi],
\ee
and
\be
f_\infty(\pi/2)=-2\ln(\sin(\pi/2))+\beta=\beta,
\ee
and the diameter of
$\mathbb{S}^1$ is $\pi$.
\end{proof}

\begin{lem}\label{lem:dinfty-ge-d_0}
The distance on pointwise limit space
is bounded below by the distances on the isometric product
in Example~\ref{ex:isom-prod}:
\be
d_\infty(p,q)\ge d_0(p,q) \quad \forall 
p,q\in \Sph^2\times \Sph^1.
\ee
\end{lem} 

\begin{proof}
This follows by taking the pointwise limit of the
estimate in Lemma~\ref{lem:d_j-ge-d_0}.
\end{proof}

When applying Lemma~\ref{LemmaDinfty-cLE2}, the following lemma will be useful to us:

\begin{lem}\label{lem:lnsinx}
For any $m\in \mathbb{N}$, there exists a constant $c_m>0$ such that 
    \be
    f_\infty(x)<x^{-1/m} +\beta
    \quad \forall x\in (0,c_m).
    \ee
\end{lem}

\begin{proof}
By the definition of $f_\infty$,
we need only prove that
there exists $c_m>0$ such that
\be
    a(x):=-2\ln(\sin x)x^{\frac{1}{m}}<1
    \quad \forall x\in (0,c_m).
\ee
We take the limit using L'hopital's rule, 
\begin{align*}
\lim_{x\to 0^+}a(x)&=-2\lim_{x\to 0^+}
\dfrac{\ln(\sin x)}{x^{-\frac{1}{m}}}
=-2\lim_{x\to 0^+}
\dfrac{(1/\sin x)\cos x}{\left(-\tfrac{1}{m}\right)
x^{\left(-\frac{1}{m}-1\right)}} 
         \\
&=2m\lim_{x\to 0^+}\left(\dfrac{x}{\sin x}\right)\cdot\lim_{x\to 0^+} \left(x^{\frac{1}{m}}\cos x\right)=0.
\end{align*}
By the definition of limit,
     \be
     \forall \epsilon>0 \,\,\exists \delta_\epsilon>0
     \textrm{ s.t. } |x-0|<\delta_\epsilon
     \implies |a(x)|<\epsilon.
     \ee
     Taking $\epsilon=1$, we choose $c=c_m=\delta_\epsilon$
     to complete the proof.
\end{proof}

\subsection{Homeomorphism}

In this section, we prove 
the homeomorphism part of
Theorem~\ref{thm:ptwise-homeo}
which we state precisely
as follows:

\begin{prop}\label{prop:ptwise-homeo}
        The identity map $F:(\Sph^2\times \Sph^1,d_0)\to(\Sph^2\times \Sph^1,d_\infty)$
        is a homeomorphism
        where $d_0$ and $d_\infty$ are defined
        as in Theorem~\ref{thm:ptwise-homeo}.
\end{prop}

Before we prove Proposition~\ref{prop:ptwise-homeo},
we state and prove three lemmas: Lemma~\ref{lem:cont-off-S},
Lemma~\ref{lem:cont-in-S},
and Lemma~\ref{lem:inv-cont}.

\begin{lem}\label{lem:cont-off-S}
    The function $F$ of Proposition~\ref{prop:ptwise-homeo} is continuous at any point $p\in(\Sph^2\times\Sph^1)\setminus S$.
\end{lem}

\begin{proof}
Fix $p=(r_p,\theta_p,\varphi_p)\notin S$.
So $r_p\in (0,\pi)$ and  
\be
\bar{d}=\bar{d}_p=\min\{r_p, \pi-r_p\}>0
\ee
By Lemma~\ref{lem:iso-Pyth}, the closed
ball in the isometric
product avoids the singular set,
\be
\bar{B}_{d_0}(p,\bar{d}_p/2)\cap S=\emptyset.
\ee
For any $\epsilon>0$, 
we choose  
\be
\delta=\delta_{\epsilon,p}
=\min\{\bar{d}_p/2,\epsilon/(1+\sigma_p+f_{\infty}(r_p)) \}>0
\ee
where
\be
\sigma_p=\sin(r_p)/\sin(r_p/2).
\ee
Consider any 
\be
q=(r_q,\theta_q,\varphi_q)\in B_{d_0}(p,\delta).
\ee
By the Lemma \ref{LemmaDinfty-cLE},
and then by the properties of $d_0$ in (\ref{eq-isom-r}) and (\ref{eq-isom-phi}) and Lemma~\ref{LemmaThetaInequality},
%and by (\ref{eq-here-finf-mono})
we have
\begin{equation}
 \begin{split}
d_\infty(p,q)&\le |r_p-r_q|+\sin(r_p)d_{\Sph^1}(\theta_p,\theta_q)+f_\infty(r_p)d_{\Sph^1}(\varphi_p,\varphi_q)\\
&\le d_0(p,q)+\sin(r_p)\dfrac{\delta}{\sin(r_p/2)}+f_\infty(r_p) d_0(p,q)\\
&\le \delta+\sigma_p\delta+f_\infty(r_p) \delta\le \epsilon. 
\end{split}
\end{equation}
\end{proof}

\begin{lem}\label{lem:cont-in-S}
The function $F$ of Proposition~\ref{prop:ptwise-homeo} is continuous at $p\in S$.
\end{lem}

\begin{proof}
Consider any $p=(r_p,\theta_p,\phi_p)\in S$, so
$r_p=0$ or $r_p=\pi$.  Recall
\be 
f_\infty(r, \theta)=-2\ln( \sin r)+\beta,
\ee
so, without loss of generality, 
we can assume $r_p=0$.  

For any $\epsilon>0$, choose 
\be
\delta=\delta_{\epsilon}=\min\left\{\tfrac{1}{2},c,\left(\tfrac{\epsilon}{7+\beta}\right)^{3/2}\right\}>0,
\ee
where $c=c_3$ is from Lemma~\ref{lem:lnsinx} with $m=3$ which guarantees that
\be \label{here:lnsinx}
f_\infty(\delta) \le \delta^{-\frac{1}{3}}+\beta
%-2 \ln(\sin \delta)\le \dfrac{1}{\delta^{\frac{1}{3}}}.
\ee

Consider any
\be
q=(r_q,\theta_q,\varphi_q)\in B_{d_0}(p,\delta).
\ee
Since $d_0$ is an isometric product metric
and $r_p=0$,
\be
r_q=|r_q-r_p|\le d_0(p,q)<\delta
\ee
and
\be \label{eq:delta-minus-rq}
|\delta-r_q|\le \delta
\ee

By the triangle inequality and
then Lemma \ref{LemmaDinfty-cLE} we have,
\begin{equation}
\begin{split}
d_{\infty}(p,q)&=d_{\infty}((0,\theta_p,\phi_p),(r_q,\theta_q,\varphi_q))\\
&\le d_{\infty}((0,\theta_p,\phi_p),(\delta,\theta_p,\phi_p))+d_{\infty}((\delta,\theta_p,\phi_p),(r_q,\theta_q,\varphi_q))\\
&\le \delta +|r_q-\delta|+\sin(\delta)d_{\Sph^1}(\theta_p,\theta_q)+f_{\infty}(\delta)d_{\Sph^1}(\varphi_p,\varphi_q)
\end{split}
\end{equation}
Since $\sin^2(\delta)\le \delta^2$ and $d_{\Sph^1}(\theta_1,\theta_2)<\pi$, we have the following estimate 
\be |\sin(\delta)d_{\Sph^1}(\theta_1,\theta_2)|\le \pi\d 
\ee
By (\ref{eq-isom-phi}), we have 
\be
d_{\Sph^1}(\varphi_1,\varphi_2)\le d_0(p,q)\le \delta.
\ee
Combining this with (\ref{here:lnsinx})
we have
\be
f_{\infty}(\delta)\,d_{\Sph^1}(\varphi_p,\varphi_q)\le
(\delta^{-1/3} +\beta)\,\delta
=\delta^{2/3}+\beta\delta
\ee
Combining these with 
(\ref{eq:delta-minus-rq})
and then using $\delta<1$, we have
\begin{equation}
d_\infty(p,q)\le  2\delta+\delta\pi+\delta^{2/3}+ \beta\delta 
\,\,\le (7+\beta)\delta^{2/3}\le \epsilon.
\end{equation}
Thus $F$ is continuous at singular points, $p\in S$.
\end{proof}

\begin{lem}\label{lem:inv-cont}
The function $F$ of Proposition~\ref{prop:ptwise-homeo}
is bijective and its inverse is Lipschitz one and thus continuous
everywhere.
\end{lem}

\begin{proof}
The function $F$ is bijective because $F(p)=p$. By Lemma~\ref{lem:dinfty-ge-d_0},
\be
d_0(F^{-1}(p),F^{-1}(q))
=d_0(p,q)\le
d_\infty(p,q).
\ee
So the inverse is Lipschitz $1$ and thus continuous. 
\end{proof}

We now prove that
$F$ is a homeomorphism:

\begin{proof}[Proof of
Proposition~\ref{prop:ptwise-homeo}]
This follows immediately from the
three lemmas:
Lemma~\ref{lem:cont-off-S},
Lemma~\ref{lem:cont-in-S},
and Lemma~\ref{lem:inv-cont}.
\end{proof}

Finally we combine the pointwise convergence
of the monotone increasing sequence
with the homeomorphism:

\begin{proof}[Proof of 
Theorem~\ref{thm:ptwise-homeo}]
This follows immediately from
Proposition~\ref{prop:ptlim-exists} and  
Proposition~\ref{prop:ptwise-homeo}.
\end{proof}

\section{Gromov-Hausdorff Convergence to the Pointed Limit Space}

In this section we review Gromov-Hausdorff (GH) convergence and then prove the following GH convergence theorem which will later be applied to prove Theorem~\ref{Thm:Main}.

\begin{thm}\label{thm:GH}
The sequence of warped product Riemannian manifolds in Example~\ref{ex:sequence} 
viewed as metric spaces,
$(\Sph^2\times \Sph^1,d_j)$, converges in the uniform and Gromov-Hausdorff sense to the pointed limit space
$(\Sph^2\times \Sph^1,d_\infty)$ which is a compact
metric space homeomorphic to the
isometric product space,
$(\Sph^2\times \Sph^1,d_0)$.
\end{thm}

\subsection{Review of Gromov-Hausdorff Convergence}
Recall the following definition of the
Gromov-Hausdorff distance between compact metric spaces, first introduced by 
Edwards in \cite{edwards1975structure} and then rediscovered and
studied extensively by Gromov in \cite{Gromov-text}.

\begin{defn}\label{defn:GH}
Given a pair of metric spaces, $(X_a,d_a)$ and $(X_b,d_b)$, the Gromov-Hausdorff distance
between them is
\be
d_{GH}\left((X_a,d_a),(X_b,d_b)\right)=
\inf\{ d_H^Z\left(f_a(X_a), f_b(X_b)\right)\}
\ee
where the infimum is taken over all common metric spaces, $(Z,d_Z)$, and over all distance preserving maps,
\be
f_a: (X_a,d_a)\to (Z,d_Z)
\textrm{ and } f_b: (X_b,d_b)\to (Z,d_Z).
\ee
Here $d_H^Z$ denotes
the Hausdorff distance,
\be
d_H^Z\left(A,B\right)
=\inf \{R \,:\, A\subset T_R(B)
\textrm{ and } B\subset T_R(A)\}
\ee
between subsets $A,B \subset Z$ which is defined
using tubular neighborhoods of a given radius $R$,
\be
T_R(A)=\{z\in Z\,:\, \exists p\in A
\,s.t.\, d_Z(p,z)<R\}.
\ee
\end{defn} 

The Gromov-Hausdorff distance is definite in the sense that
\be
d_{GH}\left((X_a,d_a),(X_b,d_b)\right)=0
\ee
iff there is a bijection, $\Psi:X_a\to X_b$,
which is distance preserving:
\be
d_b(\Psi(p),\Psi(q))=d_a(p,q)
\quad \forall p,q \in X_a.
\ee
A sequence of compact metric spaces, $(X_j,d_j)$ converges in the Gromov-Hausdorff sense to the
compact metric space $(X_\infty,d_\infty)$
iff
\be
d_{GH}\left((X_j,d_j),(X_\infty,d_\infty)\right)\to 0.
\ee

For those who would like to learn more about 
 the notion of GH convergence we recommend the book of Burago-Burago-Ivanov \cite{BBI-text} and Gromov's original book
 \cite{Gromov-metric}. The only result we need for this paper is the following theorem proven by Perales-Sormani in \cite{PS-Monotone}.

\begin{thm}[Perales-Sormani Monotone to GH]\label{thm:Riem-short}
Given a compact Riemannian manifold, $(M^m,g_0)$,
possibly with boundary, with a monotone increasing sequence of Riemannian 
metric tensors $g_j$ such that
\be\label{eq:mono-g}
g_j(V,V)\ge g_{j-1}(V,V) \qquad \forall V\in TM
\ee
with uniform bounded diameter,
\be\label{eq:diam}
\diam_{g_j}(M)\le D_0.
\ee
Then the induced length distance functions $d_j: M\times M\to [0,D_0]$
are monotone increasing and converge pointwise to
a distance function,
$d_\infty: M\times M\to [0,D_0]$
so that $(M,d_\infty)$ is
a metric space. 

If the metric space $(M, d_\infty)$ is a compact metric space,
then $d_j\to d_\infty$ uniformly and
\be\label{eq:GH}
(M,d_j) \GHto (M,d_\infty).
\ee
\end{thm}

This theorem of Perales-Sormani was inspired by early work of Sormani-Tian towards proving GH convergence of our specific monotone sequence.   It seemed that many aspects of our planned proof did not use anything more than monotonicity.  When Tian chose to leave academia, Sormani decided to prove this more general theorem with Perales and then apply it within this paper.

It is worth noting that without 
compactness of the limit space,
there need not be GH convergence to
the pointwise limit space.  See, for
example, Example 2.6 and Remark 3.8 of \cite{PS-Monotone} by Perales-Sormani.  

\subsection{Proof of GH Convergence}

With our previous results about the
pointwise limit, we are now easily able
to prove GH convergence:
 
\begin{proof}[Proof of Theorem~\ref{thm:GH}]
Since the sequence in Example~\ref{ex:sequence} is monotone increasing we can apply the Perales-Sormani Monotone to GH Convergence Theorem
(stated above as Theorem~\ref{thm:Riem-short}).  We need only show the pointwise
limit space, $(\Sph^2\times\Sph^1,d_\infty)$
is compact to obtain uniform and GH convergence of
$(\Sph^2\times\Sph^1,d_j)$
to $(\Sph^2\times\Sph^1,d_\infty)$.

By our Theorem~\ref{thm:ptwise-homeo}, we know the pointwise limit, 
$(\Sph^2\times\Sph^1,d_\infty)$, is homeomorphic to the isometric product metric space $(\Sph^2\times\Sph^1,d_0)$.
Since $\Sph^2$ and $\Sph^1$ are compact metric spaces, this isometric product space is compact. Since homeomorphisms preserve compactness, our pointwise limit space, $(\Sph^2\times\Sph^1,d_\infty)$, is also compact.   
\end{proof}

\section{The Metric Completion of the 
Smooth Limit Away from the Singular Set 
is the GH Limit}

In this section, we
consider the
extreme limit space of Example~\ref{ex:limit}
with the singular set removed,
\be
((\Sph^2\times\Sph^1)\setminus S,g_\infty)
\ee
as a smooth open Riemannian manifold with 
metric tensor,
$g_\infty$, that is the smooth limit
of $g_j$ of Example~\ref{ex:sequence}
away from the singular set, $S$.   It
has a Riemannian distance function, $d_{g_\infty}$,  
as in Definition~\ref{defn:dj}
defined using curves that avoid the
singular set. This defines 
a metric space,
\be\label{eq:open-space}
(\Sph^2\times\Sph^1\setminus S,d_{g_\infty}).
\ee

In contrast, we have the compact
metric space,
$(\Sph^2\times\Sph^1,d_{\infty})$,
which is defined in Theorem~\ref{thm:ptwise-homeo}
and proven in Theorem~\ref{thm:GH}
to be the $GH$ limit of
$(\Sph^2\times\Sph^1\setminus S,d_j)$
where $d_j=d_{g_j}$ as in Definition~\ref{defn:dj}).

Here
we review the notion of
a metric completion and prove the metric completion part of Theorem~\ref{Thm:Main}
which we state 
precisely here as
Theorem~\ref{thm:completion}.

\begin{thm}\label{thm:completion}
The metric completion of
the extreme limit space,
$(\Sph^2\times\Sph^1\setminus S,d_{g_\infty})$,
of Example~\ref{ex:limit} 
is isometric to the $GH$ limit,
$(\Sph^2\times\Sph^1,d_\infty)$,
of the sequence
$(\Sph^2\times\Sph^1,d_j)$ of
Example~\ref{ex:sequence}.
In addition the singular set $S$
inside this limit space has Hausdorff dimension $1$ and 
$\mathcal{H}^1_{d_\infty}(S)=\infty$
and
\be
{\mathcal{H}}^3_{d_\infty}(\Sph^2\times \Sph^1, d_\infty)=(2\pi)^2(2\beta+4-2\ln 4).
\ee
\end{thm}

In general the metric completions of
smooth limits away from singular sets
need not be isometric to the $GH$ limits of the
sequences.  See, for example, the cinched
examples in Allen-Sormani \cite{AS-contrasting}.

\subsection{Metric Completions}

Following Munkres's pointset topology textbook (cf. \cite{munkres2000topology}) we recall:

\begin{defn}\label{def:metric completion}
Given a metric space, $(X,d)$, the metric completion, $(\bar{X},d)$, is the unique 
complete metric space, $(Y,d_Y)$, such that there exists a
distance preserving embedding $h:X\to Y$ 
such that the closure, $\overline{h(X)}$,
includes all points in $Y$.
\end{defn}

In our setting the space $X=M\setminus S$
and the space $Y=M$ where $M=\Sph^2\times\Sph^1$
and $h:X\to Y$ is the identity map.   Our
$(Y,d_Y)$ is the GH and uniform limit space,
$(M,d_\infty)$ 
which is compact and thus complete.  
Our $(X,d_X)$ is $(M\setminus S, d_{g_\infty})$.  
To prove Theorem~\ref{thm:completion}, we need only show the identity map 
\be \label{eq:mc-dist-pres}
h: (M\setminus S, d_{g_\infty})\to (M,d_\infty)
\ee
is a distance preserving and that
\be \label{eq:mc-closure}
\forall p\in M, \exists q_i\in M\setminus S
\textrm{ such that } d_\infty(p,q_i)\to 0.
\ee

\subsection{Lemmas about Monotone Increasing Sequences}

Before proving the theorem we prove
a pair of general lemmas about monotone increasing metric tensors converging smoothly away from a singular set that can be applied in other settings as well.

\begin{lem}\label{lem:d00 length monotone}%TRUE
Given any sequence of continuous
metric tensors converging pointwise,
$g_j \to g_\infty$ on $N=M\setminus S$,
such that
\be
g_j(V,V)\le g_{j+1}(V,V)\le g_\infty(V,V)
\quad \forall V\in TN,
\ee
and given any differentiable
curve $\gamma:[0,1]\to N=M\setminus S$
we have the following equality 
    \be
\lim_{j\to\infty}L_{g_j}(\gamma)=L_{g_\infty}(\gamma).
    \ee
\end{lem}

\begin{proof}
Define the functions, $h_j:[0,1]\to [0,\infty)$,
to be
\be
h_j(t)=\sqrt{g_j(\gamma'(t),\gamma'(t))}.
\ee
By the hypotheses, each $h_j$ is a continuous function satisfying,
\be
h_j(t)\le h_{j+1}(t) \le h_\infty(t)
\qquad \forall t\in [0,1],
\ee
and $h_j\to h_\infty$ pointwise.

By definition of length $L_g$ 
we have 
\be
L_{g_j}(\gamma)
=\int_{0}^{1}h_j(t)\,dt
\le\int_{0}^{1} h_{\infty}(t)\,dt
= L_{g_\infty}(\gamma) <\infty.
\ee
By the Monotone Convergence Theorem, we have 
\begin{align*}
\lim_{j\to\infty}
\int_{0}^{1} h_j(t)\,dt
&=\int_{0}^{1}\lim_{j\to\infty}
h_j(t) \,dt\\   &=\int_{0}^{1}
h_\infty(t)\,dt. 
\end{align*}
thus
\be
\lim_{j\to\infty}L_{g_j}(\gamma)
= L_{g_\infty}(\gamma).
\ee
\end{proof}

\begin{lem}\label{lem:dinfty-le-dginfty}%TRUE
Given any sequence of continuous
metric tensors converging pointwise,
$g_j \to g_\infty$ on a connected
$N=M\setminus S$,
such that 
\be
\diam_{g_j}(M) \le D_0
\ee
and such that
\be
g_j(V,V)\le g_{j+1}(V,V)\le g_\infty(V,V)
\quad \forall V\in TM,
\ee
Then for any
any $p,q\in M\setminus S$,
we have
\be
d_\infty(p,q)\le d_{g_\infty}(p,q)
\ee
where
\be
d_\infty(p,q)=\lim_{j\to\infty}d_j(p,q)
\ee
which exists because it is a monotone increasing sequence with an upper bound,
    and
    where
\be
    d_{g_\infty}(p,q)= \inf L_{g_\infty}(\hat{c})
\ee
where the infimum is over all  continuously differentiable curves $\hat{c}:[0,1]\to M\setminus S$ such that $\hat{c}(0)=p$ and
$\hat{c}(1)=q$.
\end{lem}

Warning: as seen in the cinched examples of Allen-Sormani in \cite{AS-contrasting},
it is possible that
\be
d_\infty(p,q)< d_{g_\infty}(p,q)
\ee
if the singular set has been cinched to
have shorter lengths.  

\begin{proof}
By Definition \ref{defn:dj}, 
we know that,
\be
d_j(p,q) \le L_{g_j}(\hat{c}).
\ee
By Lemma \ref{lem:d00 length monotone},
which can be applied to our sequence of
metric tensors, $g_j$, we have  
\be
L_{g_j}(\hat{c})\le L_{g_\infty}(\hat{c}).
\ee
So we have 
\be
d_j(p,q)\le L_{g_\infty}(\hat{c}).
\ee
Taking $j\to \infty$, by the
pointwise convergence of $d_j\to d_\infty$,
we have
\be
d_\infty(p,q)\le L_\infty(\hat{c}).
\ee
We complete the proof by taking the infimum over
$\hat{c}$.
\end{proof}

%June 12: removed to bad subsections

\subsection{Avoiding Tubular Neighborhoods of Singular Sets}

In this section we show that sufficient controls on the
$g_j$-lengths of curves that lie in compact sets, $K$, that avoid the singular set, $S$, allow one to estimate $d_{g_\infty}$ and ultimately to prove $d_\infty=d_{g_\infty}$ on $M\setminus S$.  We prove two general lemmas and then prove two specific propositions related to our Example~\ref{ex:sequence} and its limit Example~\ref{ex:limit}.

\begin{lem}\label{lem:rho_j}
Suppose we have a metric space with sequence of increasing distance functions,
$(M,d_j)$ with $d_j \le d_{j+1}$,
with a uniform upper bound, $D_0$
so that $d_j \to d_\infty$ pointwise on $M\times M$.
Assume in addition that there is
a compact subset $S\subset M$
and there are functions,
\be 
\rho_j: M \to [0,D_0]
\ee
defined by
\be\label{eq:rho-j-inf}
\rho_j(p)=\inf\{d_j(p,z)\,|\, z\in S\}
\ee
then we have monotone convergence:
\be \label{eq:rho-mono}
\rho_j \le \rho_{j+1} \to \rho_\infty
\ee
pointwise on $M$,
where
\be\label{eq:rho-infty-inf}
\rho_\infty(p)=\inf\{d_\infty(p,z)\,|\, z\in S\}
\ee
Furthermore the tubular neighborhoods,
\be
T^j_R(S)=\rho_j^{-1}([0,R))
\ee
satisfy
\be\label{eq:rho-supset}
T^j_R(S) \supset
T^{j+1}_R(S) 
\supset T^\infty_R(S)
\ee
and
we have for any $j\in {\mathbb N}\cup\{\infty\}$,
\be
\bigcup_{R>0} M\setminus T^j_R(S)
=M\setminus S.
\ee
\end{lem}

\begin{proof}
Fix $p\in M$. The monotonicity $\rho_j(p)$ in \ref{eq:rho-mono}) follows immediately from $d_j \le d_{j+1}\le d_\infty$ and the definition of inf.  Since $\rho_j(p)\le D_0$, we know there is a pointwise limit, $\rho'_\infty(p)$
which
we must show equals $\rho_\infty(p)$.   Since $d_j\le d_\infty$,
taking infima on both sides we have
$\rho_j(p)\le \rho_\infty(p)$
and taking $j\to \infty$
we have
$$\rho_\infty'(p)\le \rho_\infty(p).
$$
Since $S$ is compact the inf 
in (\ref{eq:rho-j-inf})
is achieved by
$z_{p,j}\in S$,
$$
\rho_j(p)= 
d_j(p,z_{p,j}).
$$
Since $S$ is compact, we can find a subsequence
$z_{p,j}\to z_{p}$ in 
$(S,d_\infty)$.  
For any $\epsilon$, we can take $j\ge N_{\delta}$ such that
$
d_\infty(z_{p,j},z_p)<\delta,
$
By the triangle inequality,
choice of $z_{p,j}$ and $d_j\le d_\infty$,
we have
$$
d_j(p,z_p)\le
d_j(p,z_{p,j})+d_j(z_{p,j},z_p)
<\rho_j(p)+\delta.
$$
Taking $j\to \infty$
we get
$$
d_\infty(p,z_p)
\le \rho'_\infty(p)+\delta.
$$
Taking $\delta\to 0$
on the right and
applying the definition
of $\rho_\infty$ on the
left we get
$$
\rho_\infty(p)\le
\rho'_\infty(p)
$$
The second claim in (\ref{eq:rho-supset}) follows because $q\in T^\infty_R(S)$ implies there exists $z\in S$ such that $d_\infty(q,z)<R$, which implies $d_{j+1}(q,z)<R$ and the rest follows easily from there.
\end{proof}

In the following proposition we see that our sets are exceptionally simple in our example:

\begin{prop}\label{prop:rho_j}
For our particular sequence
of $g_j$ in Example~\ref{ex:sequence}
we have
\be
\rho_j(r,\theta,\phi)
=\min\{r, \pi-r\}
\quad \forall j \in {\mathbb N} \cup \{ \infty \}.
\ee
\end{prop}

\begin{proof}
This follows for all $j \in {\mathbb N}$ because
\be
S= r^{-1}(0)\cup r^{-1}(\pi)
\ee
and because
\be
g_j(\partial r, \partial r)=1.
\ee
Taking $j\to \infty$
we have $\rho_j \to \rho_\infty$
so it holds for $j=\infty$ as well.
\end{proof}

\begin{lem}\label{lem:limsup-lambda}
Suppose the hypotheses of Lemma~\ref{lem:dinfty-le-dginfty}
hold and $\rho_j:M \to [0,D_0]$
are as in Lemma~\ref{lem:rho_j}.
Let $K$ be a compact connected
set in $M\setminus S$
and let
\be
d_{j,K}(p,q)=\inf L_{g_j}(c)
\ee
where the infimum
is over $c:[0,1]\to K$
and $c(0)=p$ and $c(1)=q$.
Let
\be \label{eq:lambda-1}
\lambda_{j}(K)=\max_{p,q\in K}
|d_j(p,q) - d_{j,K}(p,q)|.
\ee  
Then for all $p,q\in K$,
\be
|d_\infty(p,q)-d_{\infty,K}(p,q)|\le \limsup_{j\to \infty} \lambda_{j}(K).
\ee
In particular, for all $p,q\in  M\setminus S$,
we have
\be\label{eq:lambda-2b}
|d_\infty(p,q)-d_{g_\infty}(p,q)|
\le 
\limsup_{j\to \infty} \lambda_{j}(K)
\ee
for any compact connected set $K$
such that $p,q\subset K \subset M\setminus S$.
In particular if the sets
$
K=M\setminus T^\infty_R(S)
$
are compact and connected with
\be
\lim_{R\to 0} \limsup_{j\to \infty} \lambda_{j}(M\setminus T^\infty_R(S))=0
\ee
then we have
\be \label{eq:lambda-3c}
d_\infty(p,q)=d_{g_\infty}(p,q)
\qquad \forall p,q\in M\setminus S.
\ee
\end{lem}

\begin{proof}
Since $g_j \to g_\infty$
on compact sets, we have
by the Perales-Sormani Monotone to GH Theorem  (see Theorem~\ref{thm:Riem-short} within) that
\be
|d_{j,K}(p,q)- d_{\infty,K}(p,q)|<\delta \quad
\forall p,q\in K, \,\,
\forall j\ge N^{g,K}_\delta
\ee
We already know 
\be
|d_j(p,q)-d_\infty(p,q)|
<\delta \quad
\forall p,q\in M,\,\,
\forall j\ge N_\delta.
\ee
Thus
for all $p,q\in K$,
\be
|d_\infty(p,q)-d_{\infty,K}(p,q)|< 2 \delta +
\lambda_{j}(K)
\ee
for all
\be
j\ge \max
\{N^{g,K}_\delta,N_\delta\}
\ee
Taking the limsup as $j\to \infty$, and then
$\delta \to 0$, we get
\be
|d_\infty(p,q)-d_{\infty,K}(p,q)|\le \limsup_{j\to \infty} \lambda_{j}(K).
\ee
The claim in (\ref{eq:lambda-2b}) follows from the fact that
\be
d_\infty(p,q)\le d_{g_\infty}(p,q)
\le d_{\infty,K}(p,q).
\ee
which follows from Lemma~\ref{lem:dinfty-le-dginfty}
and $K\subset M\setminus S$.
The final claim follows from the
exhaustion by complements of
tubular neighborhoods proven in Lemma~\ref{lem:rho_j}.
\end{proof}

We now apply the above lemma to Example~\ref{ex:sequence}.

\begin{prop}\label{prop:limsup-lambda}
For our particular sequence
of $g_j$ in Example~\ref{ex:sequence}
we have
\be
\lambda_j(M\setminus T^j_R(S))
\le 3\pi \sin(R)
\ee
where $\lambda_j(K)$
is defined as in (\ref{eq:lambda-1}).
In addition,
\be
\limsup_{j\to\infty}
\lambda_j(M\setminus T^j_R(S))
\le 3\pi \sin(R).
\ee
In particular, for all $p,q\in M\setminus S$,
\be
|d_\infty(p,q)-d_{g_\infty}(p,q)|
\le 
3\pi \sin(R)
\ee
if $r(p),r(q)\in [R,\pi-R]$ which gives
us
\be
d_\infty(p,q)=d_{g_\infty}(p,q)
\ee
\end{prop}

\begin{proof}
Fix $R\in (0,\pi/2)$ and
note that the set
$K=M\setminus T^j_R(S)$
satisfies
\be
K=
r^{-1}[R,\pi-R]
\ee
by Proposition~\ref{prop:rho_j},
and thus is compact and connected.
Consider $p,q\in K$ achieving the
max in (\ref{eq:lambda-1})
so that
\be \label{eq:lambda-1a}
\lambda_{j}(K)=
|d_j(p,q) - d_{j,K}(p,q)|.
\ee
Let $c_j:[0,1]\to M$
be a $g_j$-geodesic
with $c_j(0)=p$ and $c_j(1)=q$
that achieves
\be
L_{g_j}(c)=d_{j}(p,q).
\ee
We claim the geodesic
$c_j(0,1)$ intersects $r^{-1}(0) \cup r^{-1}(\pi)$ at either an empty set or a single point.  Suppose not.  Then it intersects at two points.
If one of the points is in $r^{-1}(0)$ and
the other is in $r^{-1}(\pi)$, then it is no longer minimizing past these points because its projection to the sphere has begun to wrap around and a more efficient geodesic can be chosen. If both lie in $r^{-1}(0)$ or both in $r^{-1}(\pi)$, then the projection of $c_j$ to the sphere has a closed loop. We could rotate this loop and its lift in the theta direction by a fixed amount without changing the length of the lift.  This would give a new curve that is still continuous but not smooth of minimal length, which is not possible in a smooth Riemannian manifold.

For any $\epsilon>0$, we can find a piecewise
smooth curve
$c_{j,\epsilon}:[0,1]\to M\setminus S$
with three smooth components
with $c_{j,\epsilon}(0)=p$ and $c_{j,\epsilon}(1)=q$
such that
\be
d_{j}(p,q)\le L_{g_j}(c_{j,\epsilon})=d_{j}(p,q)+\epsilon.
\ee
To do this we need only remove the component next to the single point in $S$ and add a smooth arc with constant $r$ to reconnect them.
Since these curves avoid $r^{-1}(0)$ and $r^{-1}(\pi)$,
we can now write
\be
c_{j,\epsilon}(s)=
(r_{j,\epsilon}(s),
\theta_{j,\epsilon}(s),\phi_{j,\epsilon}(s))
\ee
where each component is a piecewise smooth function.   Note that we can guarantee that
$\theta_{j,\epsilon}$ is monotone in $s$ on each of the three smooth components of the curve so that
\be\label{eq:Change-in-theta}
\int_0^1 |\theta_{j,\epsilon}'(s)| \, ds\le 3\pi.
\ee
Let 
\be
I_R=\{s\,:\,r_{j,\epsilon}(s)<R\}
\textrm{ and }
I^{\pi-R}=\{s\,:\,r_{j,\epsilon}(s)>\pi-R\}
\ee
and define
\be
r_{R,j,\epsilon}(s)
=
\left\{
\begin{array}{ll}
      R & s\in I_R \\
      r_{j,\epsilon}(s) & s\in [0,1]\setminus (I_R\cup I^{\pi-R}) \\
      \pi-R & s\in I^{\pi-R} 
\end{array} 
\right. 
\ee
Let
\be
c_{R,j,\epsilon}(s)=
(r_{R,j,\epsilon}(s),
\theta_{j,\epsilon}(s),\phi_{j,\epsilon}(s)).
\ee
So $c_{R,j,\epsilon}:[0,1]\to
r_j^{-1}[R,\pi-R]=K$
and $c_{R,j,\epsilon}(0)=p$
and $c_{R,j,\epsilon}(1)=q$
so
\be
d_{j,K}(p,q)\le L_{g_j}(
c_{R,j,\epsilon}).
\ee
Thus by (\ref{eq:lambda-1a})
and $d_{j,K}(p,q)\ge d_{j}(p,q)$, we know
\begin{eqnarray}
\lambda_{j}(K)
&\le& L_{g_j}(c_{R,j,\epsilon})
-L_{g_j}(c_{j,\epsilon})\\
&\le &
\int_0^1
|c_{R,j,\epsilon}'(s)|_{g_j}\,ds
-
\int_0^1
|c_{j,\epsilon}'(s)|_{g_j}\,ds \\
&\le&
\int_{I_R\cup I^{\pi-R}}
|c_{R,j,\epsilon}'(s)|_{g_j}-|c_{j,\epsilon}'(s)|_{g_j} \,\, ds.
\end{eqnarray}
The last line above follows from the fact that these curves agree 
on $[0,1]\setminus (I_R\cup I^{\pi-R})$.
By the definition of $g_j$
in Example~\ref{ex:sequence}
we have on $I_R\cup I^{\pi-R}$
\begin{eqnarray}
|c_{R,j,\epsilon}'(s)|_{g_j}&=&\sqrt{a_R^2(s)+b_R^2(s)+u_R^2(s)}\\
|c_{j,\epsilon}'(s)|_{g_j}&=&\sqrt{a^2(s)+b^2(s)+u^2(s)}
\end{eqnarray}
where
\begin{eqnarray}
a_R(s)&=&0\\
b_R(s)&=& \sin(R)|\theta_{j,\epsilon}'(s)|\\
u_R(s)&=& f_j(R)|\phi_{j,\epsilon}'(s)|\\
a(s)&=& |r_{j,\epsilon}'(s)|\\
b(s)&=& \sin(r_{j,\epsilon}(s))\,|\theta_{j,\epsilon}'(s)|\\
u(s)&=& f_j(r_{j,\epsilon}(s))\,|\phi_{j,\epsilon}'(s)|.
\end{eqnarray}
So 
\be
a_R(s)\le a(s)\qquad
b_R(s) \ge b(s) \qquad u_R(s) \le u(s)
\ee
and thus
\begin{eqnarray}
|c_{R,j,\epsilon}'(s)|_{g_j}&\le&\sqrt{a^2(s)+b_R^2(s)+u^2(s)}\\
|c_{j,\epsilon}'(s)|_{g_j}&\ge &\sqrt{a^2(s)+b^2(s)+u^2(s)}.
\end{eqnarray}
Taking the difference of these two norms
and applying the triangle inequality we get
\be
|c_{R,j,\epsilon}'(s)|_{g_j}-
|c_{j,\epsilon}'(s)|_{g_j} \le |b_R(s)-b(s)|
\le |b_R(s)|.
\ee
So we can we can estimate the integral
using (\ref{eq:Change-in-theta}) as follows:
\begin{eqnarray}
\lambda_{j}(K)
&\le& \int_{I_R\cup I^{\pi-R}}
\sin(R) \,|\theta_{j,\epsilon}'(s)| \, ds\\
&\le & \sin(R) \int_0^1 |\theta_{j,\epsilon}'(s)| \, ds \le \sin(R)\,3\pi.
\end{eqnarray}
The rest of the claims easily follow.
\end{proof}

\subsection{Metric Completion Proof}

In this section we prove a general lemma about metric completions and
then a proposition specific to our Example~\ref{ex:limit} and then complete the proof of Theorem~\ref{thm:completion}.

\begin{lem}\label{lem:closure}
Suppose $(M, d_\infty)$ is compact and
\be
d_{g_\infty}(p,q)=d_\infty(p,q)
\quad
\forall p,q \in M\setminus S.
\ee
Suppose that for all $z\in S$ there exists
$p_i\in M\setminus S$ such that
$d_\infty(p_i,z)\to 0$.
Then
$(M,d_\infty)$ is the metric completion of $(M \setminus S, d_{g_\infty})$.   
In particular, this holds when
$S$ is compact.
\end{lem}

\begin{proof}
This follows immediately from
Definition~\ref{def:metric completion}.
This can be done
whenever $S$ is compact
because we saw in Lemma~\ref{lem:rho_j} that
\be
M\setminus S=\bigcup_{R>0}
\, \bigg(M\setminus T^\infty_R(S)\bigg).
\ee
and we can take $R=1/i$.
\end{proof}

\begin{prop}\label{prop:closure}
In Examples~\ref{ex:sequence} and~\ref{ex:limit}
for all $z\in S$ we have $p_i\in M\setminus S$ such that
\be
d_j(p_i,z)=1/i \qquad \forall j \in {\mathbb N}.
\ee
and so the pointwise limit has
\be
d_\infty(p_i,z)=1/i .
\ee
\end{prop}

\begin{proof}
If $z\in S$ then 
\be
z=(0, \theta_z,\phi_z) 
\textrm{ or }z=(\pi,\theta_z, \phi_z)
\ee
respectively.
We choose
\be
p_i=(1/i, \theta_z,\phi_z) 
\in M\setminus S
\textrm{ or }p_i=(\pi-1/i,\theta_z, \phi_z)\in M\setminus S,
\ee
respectively.
By 
Lemma~\ref{LemmaDinfty-cLE},
we have
\be
d_j(p_i,z)\le 1/i.
\ee
Taking the limit as $j\to \infty$
we have our claim.
\end{proof}

\subsection{Hausdorff Measure and Dimension of the Singular Set in the Limit Space:}
\label{subsect-Hausdorff} 

Here we prove the following  propositions regarding the Hausdorff measure and dimension of the singular set.

\begin{prop}\label{prop:Hausdorff-Dimension}
In $(\Sph^2\times \Sph^1, d_\infty)$ as in Theorem~\ref{thm:ptwise-homeo}, the Hausdorff dimension of the singular set $S\subset (\Sph^2\times \Sph^1)$ is $1$.
In particular, $\mathcal{H}^3_{d_\infty}(S)=0$.
Furthermore
$\mathcal{H}^1_{d_\infty}(S)=\infty$.
\end{prop}

Before we can prove this proposition
we need the following lemmas.

%First we have a basic calculation:

%WE DON'T NEED THIS
%\begin{lem}\label{lem:quantitative estimate}
%There exists uniform $c_{Ln}\in (0,\pi)$ such that 
%\be
%-2\ln \sin x\ge 2x \qquad \forall x\in (0,c_{Ln}].
%\ee
%and for $f_\infty$ as in Example~\ref{ex:limit}
%we have
%\be
%f_\infty(x) \le \beta+2x
%\qquad \forall x\in (0,c_{Ln}].
%\ee
%\end{lem}
% REMOVED THE PROOF SAT JUN 14

First we  control the distances between points in $r^{-1}(0)$
in the following lemma:

\begin{lem}\label{lem:Distance-in-S}
For any $m\in \mathbb N$,
we can choose
\be
\d_m=\min\{c_m,\tfrac{\pi}{2}\}
\ee
where $c_m$ is the constant in Lemma \ref{lem:lnsinx}, 
%and where $c_{Ln}$ is the constant in Lemma \ref{lem:quantitative estimate} 
with the following property:
For any $\varphi_1,\ \varphi_2\in [0, 2\pi)$ satisfying
\be\label{eq:phi-d}
d_{\Sph^1}(\varphi_1,\varphi_2)\le \delta<\d_m
\ee
we have
\be
d_{\infty}((0,0,\varphi_1),(0,0,\varphi_2))\le (3+\beta)\,\d^{1-\frac{1}{m}}.
\ee
\end{lem}

\begin{proof}
By Lemma~\ref{LemmaDinfty-cLE}, we have for all $r\in (0,\pi/2)$
\be \label{eq:use-LE}
d_{\infty}((0,0,\varphi_1),(0,0,\varphi_2))\leq 2 r+f_\infty(r) \,|\varphi_1-\varphi_2|
\le 2r +f_\infty(r)\,\d.
\ee
Choosing 
\be
r= d_{\Sph^1}(\varphi_1,\varphi_2)\le \delta
\ee
and applying $\delta<\pi/2$, 
we have
\be
d_{\infty}((0,0,\varphi_1),(0,0,\varphi_2)) \le 2\delta+f_\infty(\delta)\,\delta.
\ee
Since
$\delta<\delta_m\le c_m$, by Lemma \ref{lem:lnsinx}, 
\be
d_{\infty}((0,0,\varphi_1),(0,0,\varphi_2)) \le 2\delta+
(\delta^{-1/m} +\beta)\,\delta.
\ee
This gives our claim because
$\delta<1/2$ implies
$\d^{1/m}<1$ so
$
\delta\le \d^{1-\frac{1}{m}}
$
and so
\be
2\delta+(\delta^{-1/m} +\beta)\,\delta
\le
2\d^{1-\frac{1}{m}}+\d^{1-\frac{1}{m}}+\beta\,\d^{1-\frac{1}{m}}=(3+\beta)\,\d^{1-\frac{1}{m}}.
\ee
\end{proof}

We apply this lemma to prove the
following:

\begin{lem}\label{lem:Hausdorff-0}
Under the hypotheses of
proposition~\ref{prop:Hausdorff-Dimension}, the
 $p$ dimensional Hausdorff measure of $S$ in $(\Sph^2\times \Sph^1,d_\infty)$, denoted
 $\mathcal{H}^p(S)$, 
 is zero for all $p>1$. 
\end{lem}

\begin{proof}
 We need only prove that the
 $p$ dimensional Hausdorff measure of $S$, denoted
 $\mathcal{H}^p(S)$, is zero for all $p>1$.  
 Since $S=r^{-1}(0)\cup r^{-1}(\pi)$, and
 we have symmetry, we need only show
 \be
 \mathcal{H}^p(r^{-1}(0))=0
 \qquad \forall p>1.
 \ee
 Recall that 
 \be
 r^{-1}(0)=\{(0,0,\phi)\,:\,\phi\in \Sph^1\}\subset \Sph^2\times \Sph^1.
 \ee

For any $p>1$, there exist 
a natural number $m=m_p>1$ 
such that
\be\label{eq:choose-mp}
1+\tfrac{1}{m-1}<p
\ee

For any $\delta<\d_m$,
where $\d_m$ is defined as in Lemma \ref{lem:Distance-in-S},
we choose $N=N_\delta\ge N_m$ 
sufficiently large
that
\be
2\pi/N\le \delta
\ee 
.

Choose a collection of evenly spaced
points 
\be
p_i=(0,0,2\pi i/N) \in r^{-1}(0).
\ee
We can estimate
the Hausdorff measure,
$
\mathcal{H}^p(r^{-1}(0))=
\lim_{\delta\to 0}
\mathcal{H}^p_\delta(S)
$,
using balls of radius $\le \delta$
about these evenly spaced points.

By Lemma \ref{lem:Distance-in-S},
\be\label{eq:r_N-above}
r_N= d_\infty(p_i,p_{i+1})/2
\le (3+\beta)\left(2\pi N^{-1}\right)^{1-\frac{1}{m}}
\ee
Then the
closed $d_\infty$-balls of radius
$r_N$
centered on these points 
cover $r^{-1}(0)$:
\be
r^{-1}(0)\subset \bigcup_{j=1}^{N}
\bar{B}_{p_j}(r_N).
\ee
Thus there is a constant, $C_p$ such that
\be
\mathcal{H}_{\delta}^p(r^{-1}(0))
\le C_p \sum_{i=1}^N r_N^p
=C_p N r_N^p.
\ee 

By our estimate on $r_N$ in (\ref{eq:r_N-above}),
\begin{eqnarray}
\mathcal{H}_{\delta}^p(r^{-1}(0))
&\le&
C_p N (3+\beta)^p \left(2\pi N^{-1}\right)^{(1-\tfrac{1}{m})p}\\
&\le&
C_p (3+\beta)^p
\left(2\pi\right)^{(1-\tfrac{1}{m})p}
N^{1-p+p/m}.
\end{eqnarray}
To estimate the Hausdorff measure we
must take $\delta\to 0$ which requires us to take $N\to \infty$.

We need only show that
\be
\lim_{N\to \infty}
N^{1-p+\tfrac{p}{m}}=0.
\ee
So we need only show
\be
0>1-p+\tfrac{p}{m}.
\ee
Since our choice of
$m=m_p$ at the top
in (\ref{eq:choose-mp}),
is positive, we need only show
\be
0>m-mp+p
\ee
which follows from
$mp-p>m$ and $p(m-1)>m$.
Since our choice of $m=m_p>1$,
we need only show
\be
p>m/(m-1)
\ee
which follows from
our choice of
$m=m_p$ 
in (\ref{eq:choose-mp}).
 Hence $\mathcal{H}^p(S)=0$ for all $p>1$.
\end{proof}

\begin{lem}\label{lem:Hausdorff-infty}
Under the hypotheses of Proposition~\ref{prop:Hausdorff-Dimension},
the singular set, $S$, has Hausdorff dimension one in $(\Sph^2\times \Sph^1,d_\infty)$.   
However
\be\label{eq:unrect}
\mathcal{H}^1_{d_\infty}(r^{-1}(0))=\infty
\textrm{ and }
\mathcal{H}^1_{d_\infty}(r^{-1}(\pi))=\infty.
\ee
In fact the continuous curves,
 $c_r:\Sph^1\to r^{-1}(0)$,
 defined by
 $c_r(t)=(r,0,t)$ for $r=0$
 and for $r=\pi$
 are unrectifiable bijections.
\end{lem}

\begin{proof}
We know that for $r=0$ and $r=\pi$
that $c_r$ is a bijective homomorphism, by Proposition~\ref{prop:ptwise-homeo}.
By Lemma~\ref{lem-fibres},
we know
\be
L_{g_j}(c_r)=2\pi f_j(0)
\ee
which means $c_r$ is $d_j$-rectifiable
and
\be
2\pi f_j(0)=\sup \sum_{i=1}^N
d_j(c_r(t_{i-1}),c_r(t_i))
\ee
where the supremum is over all
$N\in {\mathbb N}$ and over all
partitions,
\be
P=\{t_0,t_1,...,t_N\}
\ee
such that
\be
0=t_0<t_1<\cdots< t_N=2\pi.
\ee
Thus for each $j\in {\mathbb N}$,
there is $N_j\in {\mathbb N}$,
and a partition, 
\be
P_j=\{t_{j,0},t_{j,1},...,t_{j,N}\}
\ee
such that
\be
\pi f_j(0)< \sum_{i=1}^N
d_j(c_r(t_{j,i-1}),c_r(t_{j,i})).
\ee
Since $d_\infty(p,q)\ge d_j(p,q)$
for all $p,q\in \Sph^2\times\Sph^1$
we have
\be
\pi f_j(0)< \sum_{i=1}^N
d_\infty(c_r(t_{j,i-1}),c_r(t_{j,i})).
\ee

We claim $c_r$ is not $d_\infty$-rectifiable.  Assume on the
contrary that $c_r$
has finite $d_\infty$-rectifiable length, $L_{d_\infty}(c_r)$. Then
it is larger than the sum and we have
\be
\pi f_j(0)<L_{d_\infty}(c_r).
\ee
However $\lim_{j\to \infty}f_j(0)=\infty$, so this is a 
contradiction.

Since $c_0$ and $c_\pi$ are
$d_\infty$-unrectifiable bijections 
their images have unbounded Hausdorff measure of dimension one as claimed
in (\ref{eq:unrect}).
Thus
${\mathcal H}^1(S)=\infty$.
Combining this with Lemma~\ref{lem:Hausdorff-0},
we see that the Hausdorff dimension
of the singular set is $1$.
\end{proof}

We can now prove Proposition~\ref{prop:Hausdorff-Dimension}:

\begin{proof}
This follows immediately from
Lemma~\ref{lem:Hausdorff-0}
and 
Lemma~\ref{lem:Hausdorff-infty}.
\end{proof}

\subsection{Volume of the Limit Space}

We can now find the volume of the limit space:

\begin{prop} \label{prop:Hausdorff-vol-limit}
The metric space $(\Sph^2\times\Sph^1, d_\infty)$ as in Theorem~\ref{thm:ptwise-homeo} has 
volume,
\be
{\mathcal{H}}^3_{d_\infty}(\Sph^2\times \Sph^1, d_\infty)=(2\pi)^2(2\beta+4-2\ln 4).
\ee
\end{prop}

\begin{proof}
Since ${\mathcal H}^3_{d_\infty}(S)=0$
we know
\be
{\mathcal{H}}^3_{d_\infty}(\Sph^2\times \Sph^1, d_\infty)=
{\mathcal{H}}^3_{d_\infty}((\Sph^2\times \Sph^1)\setminus S, d_\infty).
\ee
By Proposition~\ref{prop:limsup-lambda} we have
\be
d_\infty(p,q)=d_{g_\infty}(p,q)
\quad \forall p,q \in (\Sph^2\times \Sph^1)\setminus S.
\ee
so
\be
{\mathcal{H}}^3_{d_\infty}((\Sph^2\times \Sph^1)\setminus S, d_\infty)=
\vol_{g_\infty}((\Sph^2\times \Sph^1)\setminus S).
\ee
By Lemma 3.3 in \cite{STW-Extreme} by Sormani-Tian-Wang which is proven
by integration, we have 
\be
\vol_{g_\infty}((\Sph^2\times \Sph^1)\setminus S)=(2\pi)^2(2\beta+4-2\ln 4).
\ee
\end{proof}

\subsection{Proof of our Main Theorem}

We can now prove Theorem~\ref{thm:completion}
and then apply it to complete the proof 
of Theorem~\ref{Thm:Main}

\begin{proof}[Proof of Theorem~\ref{thm:completion}]
In Theorem~\ref{thm:ptwise-homeo}
we saw that $(\Sph^2\times \Sph^1,d_\infty)$ is homeomorphic to the isometric product, $\Sph^2\times \Sph^1$, and is thus compact.
In Theorem~\ref{thm:GH} we proved
$(\Sph^2\times \Sph^1,d_\infty)$
is the GH limit of $(\Sph^2\times \Sph^1,d_j)$.
In Proposition~\ref{prop:limsup-lambda} we saw that
\be
d_\infty(p,q)=d_{g_\infty}(p,q)
\quad \forall p,q \in (\Sph^2\times \Sph^1)\setminus S.
\ee
In Proposition~\ref{prop:closure} we saw that all $p\in \Sph^2\times\Sph^1$
lie in the closure of $(\Sph^2\times \Sph^1)\setminus S$.   
Thus by Definition~\ref{def:metric completion}
$(\Sph^2\times \Sph^1, d_\infty)$
is the metric completion of
$((\Sph^2\times \Sph^1)\setminus S, d_{g_\infty})$.
The two claims regarding Hausdorff measures follow from
Proposition~\ref{prop:Hausdorff-vol-limit}
and Proposition~\ref{prop:Hausdorff-Dimension}.
\end{proof}

We can now prove Theorem~\ref{Thm:Main} by combining our three theorems:

\begin{proof}[Proof of Theorem~\ref{Thm:Main}]
This is an immediate consequence of Theorem~\ref{thm:ptwise-homeo}, Theorem~\ref{thm:GH},
and Theorem~\ref{thm:completion}
proven above. 
\end{proof}

\section{Intrinsic Flat Convergence}

In this section, we prove the
volume preserving intrinsic flat ($\mathcal{VF}$)
convergence of our sequence.
Intrinsic flat convergence of oriented Riemannian manifolds and more general integral current spaces was introduced by
by Sormani-Wenger in \cite{SW-JDG} using the deep theory of integral currents and Lipschitz tuples on metric spaces developed by Ambrosio-Kirchheim in \cite{AK}.

\begin{thm}\label{thm:VF}
If we view the sequence of Riemannian
manifolds $(\Sph^2\times \Sph^1,g_j)$
of Example~\ref{ex:sequence}
as integral current spaces,
\be
(\Sph^2\times \Sph^1, d_j, [[\Sph^2\times \Sph^1]]),
\ee
then they converge in the volume preserving intrinsic flat sense to 
$
(M_\infty, d_\infty, T_\infty)
$
where $d_\infty$ is the distance on the
Gromov-Hausdorff limit, $(\Sph^2\times \Sph^1,d_\infty)$, and where
\be \label{eq:set-Riem-1}
M_\infty=\set_{d_\infty}(T_\infty)\subset \Sph^2\times \Sph^1 \textrm{ with closure }
\overline{M}_\infty=\Sph^2\times \Sph^1.
\ee
and where
$T_\infty$ is an 
integral current on $(\Sph^2\times \Sph^1,d_\infty)$ such
that $T_\infty=[[\Sph^2\times \Sph^1]]$ viewed
as an integral current on $(M,d_j)$ which means
for any $j \in {\mathbb N}$:
\be
T_\infty(\pi_0,...,\pi_3)=\int_{\Sph^2\times \Sph^1\setminus S}
\pi_0 d\pi_1\wedge\cdots \wedge \pi_3
\ee
for any tuple of Lipschitz functions,
\be
\pi_i: (\Sph^2\times \Sph^1, d_j)\to \mathbb{R}
\ee
where $\pi_0$ is bounded.
\end{thm}

We begin with a review of
$\mathcal{VF}$ convergence
and integral currents and
then prove key lemmas before
completing the proof of this theorem.

\subsection{A Theorem about Intrinsic Flat Convergence}

Due to limited space, we have no room to review the Geometric Measure Theory needed to truly understand
this part of the paper.   

For those who would like to 
learn enough about the notion of volume preserving intrinsic flat ($\mathcal{VF}$) convergence and integral current spaces to understand the statements of our Theorem~\ref{thm:VF}
and the theorem below, we recommend reading the background in the paper of Perales-Sormani \cite{PS-Monotone}.  
The only result about volume preserving intrinsic flat convergence that we need for this paper is the
following theorem proven by Perales-Sormani in \cite{PS-Monotone}:

\begin{thm}[Perales-Sormani-Monotone-Convergence] \label{thm:Riem}
Given a compact Riemannian manifold, $(M^m,g_0)$,
possibly with boundary, with a monotone increasing sequence of Riemannian 
metric tensors $g_j$ such that
\be
g_j(V,V)\ge g_{j-1}(V,V) \qquad \forall V\in TM
\ee
with uniform bounded diameter,
\be
\diam_{g_j}(M)\le D_0.
\ee
Then the induced length distance functions $d_j: M\times M\to [0,D_0]$
are monotone increasing and converge pointwise to
a distance function,
$d_\infty: M\times M\to [0,D_0]$
so that $(M,d_\infty)$ is
a metric space. 

If the metric space $(M, d_\infty)$ is a compact metric space,
then $d_j\to d_\infty$ uniformly and
\be
(M,d_j) \GHto (M,d_\infty).
\ee

If, in addition, $M$ is an oriented Riemannian manifold 
with uniformly bounded total volume,
\be \label{eq:vol}
\vol_j(M)\le V_0 \textrm{ and }\vol_j(\partial M)\le A_0
\ee
then we have volume preserving intrinsic flat convergence
\be \label{eq:VF-Riem}
(M,d_j, [[M]]) \VFto (M_\infty,d_\infty,T_\infty).
\ee
where 
``$T_\infty=[[\Sph^2\times \Sph^1]]$ viewed
as an integral current structure
on $(\Sph^2\times \Sph^1,d_j)$"
and $M_\infty$ satisfies
\be \label{eq:set-Riem}
M_\infty=\set_{d_\infty}(T_\infty)\subset M \textrm{ with closure }
\overline{M}_\infty=M.
\ee
This means
\be
T_\infty(\pi_0, \pi_1,...,\pi_3)
=\int_{\Sph^2\times \Sph^1}
\pi_0\, d\pi_1\wedge d\pi_2\wedge d\pi_3
\ee
for any tuple of Lipschitz functions
\be\label{eq:Lip-j}
\pi_i: (\Sph^2\times \Sph^1,d_j)\to {\mathbb R}
\ee
with $\pi_0$ bounded.   
\end{thm}

Note that we have already applied the top part of this theorem, which we called the Perales-Sormani Monotone to GH Theorem (stated as Theorem~\ref{thm:Riem-short} above) to prove GH convergence of our sequence.   We avoided the full statement then because intrinsic flat convergence part is somewhat complicated, and we wanted to ensure the first part of the paper was easy to ready for those only interested only in GH limits.   

This section is more difficult but can be followed easily after reading the background section of \cite{PS-Monotone} by Perales-Sormani.   One might even read that entire paper first to truly understand what is proven here.   If one accepts Perales-Sormani-Monotone-Convergence Theorem as a black box, then you can still read the proof below.

\subsection{Proof of $\mathcal{VF}$ Convergence to the Extreme Limit Space}

In this section we will prove
the intrinsic flat convergence.

\begin{proof}[Proof of Theorem~\ref{thm:VF}]
We have already proven every hypothesis in the first paragraph of the Perales-Sormani Monotone Convergence Theorem (stated as
Theorem~\ref{thm:Riem} above)
to prove our Gromov-Hausdorff convergence in Theorem~\ref{thm:GH}.  Sormani-Tian-Wang proved a uniform upper bound on $\vol_{g_j}(\Sph^2\times\Sph^1)$
in \cite{STW-Extreme} (which we reviewed in Proposition~\ref{prop-vol-sequence}).  Our manifold,
$\Sph^2\times\Sph^1$, has no boundary.  Thus, we have all the hypotheses of Theorem~\ref{thm:Riem}.  So we conclude volume preserving intrinsic flat convergence to 
\be
(M_\infty, d_\infty, T_\infty)
\ee
where
$M_\infty$ satisfies (\ref{eq:set-Riem}) and
``$T_\infty=[[\Sph^2\times \Sph^1]]$ viewed
as an integral current structure
on $(\Sph^2\times \Sph^1,d_j)$". 
\end{proof}

%%%%%%%%%%

\section{Testing Notions of Nonnegative Scalar Curvature}
\label{Sect-Open}

Now that we have a strong understanding of the metric properties of our extreme limit space, this space can be used to test various proposed notions of generalized nonnegative scalar curvature.  Here we suggest two particular notions to check, but others may be checked on this extreme example as well.  

\begin{rmrk} It would be interesting to explore whether our extreme limit space of Example~\ref{ex:limit} has nonnegative scalar curvature in the sense of prism rigidity as studied by Gromov and Chao Li in 
\cite{Gromov-Dirac} and
\cite{li2019positive}.   Since it is known that the extreme limit space is not flat anywhere, one would need only show there are no prisms satisfying the hypotheses of Gromov's prism property.   Note that it has already been shown there are no such prisms in the sequence of Example~\ref{ex:sequence} because they have positive scalar curvature, so this might be helpful towards completing a proof.   The fact that we have $C^0$ convergence on some regions should also be helpful.
\end{rmrk}

\begin{rmrk}
It would be interesting to see if one could define a Ricci flow starting from our extreme limit space of Example~\ref{ex:limit} and see if it flows 
into smooth spaces of positive scalar curvature.   One way to do this might be to find the Ricci flow of the
sequence of Riemannian manifolds in Example~\ref{ex:sequence} which definitely flow into smooth manifolds of positive scalar curvature.   Then try to take a limit of these flows and show they converge smoothly for $t>0$ even though
they only converge in the GH or SWIF or uniform sense as $t\to 0$.   If this works, then our extreme limit space has
generalized nonnegative scalar curvature in the sense defined by Bamler and Burkhardt-Guim in \cite{Bamler-Gromov} and \cite{Burkhardt_Guim_2019}
\end{rmrk}

%%%%%%%%%%%

\bibliographystyle{plain}
\bibliography{2023-Sormani.bib}

\end{document}